%% file: main_arxiv.tex
\documentclass[letterpaper, 10pt]{ieeeconf}      

\IEEEoverridecommandlockouts                              

\usepackage{graphicx}

\title{\LARGE \bf
Analysis of a continuous opinion and discrete action dynamics model coupled with an external dynamics$^*$\thanks{$^*$This work was partially supported by ANR through the grant NICETWEET No. ANR-20-CE48-0009 and the PNRR project DECIDE No. 760069.}
}

\author{A. Couthures$^1$\thanks{$^1$ Universit\'e de Lorraine, CNRS, CRAN, F-54000 Nancy, France.}, T. Mongaillard$^1$, V.S. Varma$^{1,2}$\thanks{$^2$associated with Automation Department, Technical University of Cluj-Napoca, Memorandumului 28, 400114 Cluj-Napoca, Romania.}, S. Lasaulce$^{3,1}$\thanks{$^3$ associated with Khalifa University, Abu Dhabi, UAE.}, I.C. Mor\u{a}rescu$^{1,2}$}

\input{packages}

\input{basic_macros}

\input{notation}

\begin{document}

\maketitle
\thispagestyle{empty}
\pagestyle{empty}

\begin{abstract}
We consider a set of consumers in a city or town (who thus generate pollution) whose opinion is governed by a continuous opinion and discrete action (CODA) dynamics model. This dynamics is coupled with an observation signal dynamics, which defines the information the consumers have access to regarding the common pollution. We show that the external observation signal has a significant impact on the asymptotic behavior of the CODA model. When the coupling is strong, it induces either a chaotic behavior or convergence towards a limit cycle. When the coupling is weak, a more classical behavior characterized by local agreements in polarized clusters is observed. In both cases, conditions under which clusters of consumers don't change their actions are provided.Numerical examples are provided to illustrate the derived analytical results.
\end{abstract}

\section{Introduction}

Opinion dynamics (OD) over social networks attracted a lot of attention during the last decades. Multi-agent systems have provided an efficient way to model opinion evolution under social interactions. The existing OD models consider that the opinions evolve either in a discrete set \cite{IsingThesis1924,CliffordSudbury1973,Sznajd2000,VCM2020} or in a continuous set of values \cite{Deffuant2000,krause2002,Friedkin,MG10}. While some models naturally lead to consensus \cite{DeGroot,Friedkin} some others yield a network clustering \cite{Deffuant2000,krause2002,MG10,Altafini}. However, all the models enumerated above consider that each individual has access to the opinion values of the neighbors. In order to more accurately describe the opinion dynamics and to recover more realistic behaviors, a mix of continuous opinion with discrete actions (CODA) was proposed in \cite{Martins2008}. This model reflects the fact that even if we often face binary choices or actions that are visible to our neighbors, our opinion evolves in a continuous space of values that are not accessible. A consensus-like dynamics reproducing this behavior has been proposed and analyzed in \cite{NilCoSa2016} where the preservation and the propagation of actions are also characterized through the notion of robust polarized clusters. While the model in \cite{NilCoSa2016} led to a clustering of the network, a similar idea was employed in \cite{CeragioliFrascaECC2015} to study the emergence of consensus under quantized all-to-all communication. 

In this paper we analyze the behavior of the CODA model introduced in \cite{NilCoSa2016}  coupled with an external dynamics. 
Many models have been developed to characterize the pollution dynamic in urban areas, considering the fluid dynamics approach \cite{memonEffectsBuildingAspect2010}, chemistry-based approach \cite{berkowiczOSPMParameterisedStreet2000}, or both \cite{bakerStudyDispersionTransport2004}. Even if the time constants depend on the chemical compound considered \cite{gargComprehensiveStudyImpact2021}, we introduce a simple linear pollution model to estimate the local air quality.
In this model, the pollution level depends on the actions of the individuals which in turn are influenced both by the actions of their neighbors and the pollution level. The coupling of the two dynamics leads to a complex asymptotic behavior that can be summarized as follows. When the coupling between the dynamics is weak, one recovers the asymptotic behavior of the original CODA model in \cite{NilCoSa2016}. A strong coupling between the two dynamics hampers the convergence towards a steady state and yields either chaotic oscillations or convergence towards a limit cycle. It is noteworthy that even in the simplified case when all the agents have the same initial opinion, the strong coupling with the external dynamics hampers the convergence toward a steady state and may lead to chaotic oscillations.

The main contributions of this paper are: i) the introduction of a mathematical model capturing the coupling between the CODA dynamics and an external one; ii) the analysis of the asymptotic behavior of the aforementioned model; iii) and the characterization of the coupling strength leading to different asymptotic behaviors.

The paper is structured as follows. Section \ref{sec:problem_form} presents the definitions of the measures that constitute the model. Characteristics of opinion equilibrium and asymptotic behavior are analyzed in Section \ref{sec:analysis}, followed by a focus on the synchronized behavior in Section \ref{sec:synchro_behav}. Section \ref{sec:num_results} illustrates the different behaviors with numerical simulations. Finally, Section \ref{sec:ccl} concludes our work.

\section{Problem formulation and preliminaries}\label{sec:problem_form}

\subsection*{Model description}

We consider the classical multi-agent modeling in which $n$ individuals/agents belong to the set $\Vcal = \Vset$ and interact according to a fixed graph $\Gcal = \left(\Vcal, \Ecal\right)$ that can be directed or not. The neighborhood of the agent $i$ is denoted by $\SetN{i}$ and represents the set of agents that influence $i$ according to the graph $\Gcal$ (i.e $j \in \SetN{i} \Leftrightarrow (j,i) \in \Ecal$). We also denote by $n_{i}$ the cardinality of $\SetN{i}$. We assign to each agent $i \in \Vcal$ an opinion $\opinion{i} \in \left[-1,1\right]$ that evolves in time according to a discrete time protocol defined further in \eqref{eq:opinion_dynamic}. Let $\opiniont{i}$ be the opinion of the agent $i \in \Vcal$ at time $k$ and $\opiniont{}$ the opinion of all individuals at time $k$. Let us also introduce the action value $\actiont{i} \in \SetAction$ as a quantized version of $\opiniont{i}$ defined by
\begin{equation*}
	\action{i}(k) = 
	\begin{cases} 
		1 \hspace{-2mm}& \text{if } (\opiniont{i} > 0) \lor (\opiniont{i} = 0 \wedge \action{i}(k-1)=1)\\
		-1 \hspace{-2mm}& \text{if } (\opiniont{i} < 0)\lor (\opiniont{i} = 0 \wedge \action{i}(k-1)=-1) 
	\end{cases}
\end{equation*} 

We assume that the action of individual $i$ at time $k$ generates an emission $\emissiont{i} \in \left[\emin, \emax\right] \subset \R$ where $\emin$ and $\emax$ are the minimum and maximum emissions, respectively. The emission is given by the following equation 

\begin{equation}\label{eq:emissions}
	\emissiont{i} = \begin{cases}
\emin \mbox{ if } \action{i}(k) =-1\\
\emax \mbox{ if } \action{i}(k) =1
\end{cases}.
\end{equation}

We add an external state $\state \in \R$, referred to as pollution, that captures the environment state under the emission of everyone. The pollution evolves according to the following discrete-time dynamics:

\begin{equation}\label{eq:state_dynamic}
	\statetplusun = \gamma \statet + \sum_{i=1}^{N} \emissiont{i},
\end{equation} 
where $\gamma \in (0, 1)$ is an autonomous decay rate.

We assume that individuals cannot observe $\statet$ but they can sense a quantized value $\actiont{p} \in \SetAction$. Let us define $\actiont{p}$ as a function of a threshold $\statethreshold \in \R$ as follows: 

\begin{equation*}
	\actiont{p} = 	\begin{cases} 
		-1 \hspace{-2mm}& \text{if } (\statet> \statethreshold)\lor (\statet= \statethreshold \wedge q_p(k-1)=-1) \\
		1 \hspace{-2mm}& \text{if } (\statet < \statethreshold)\lor (\statet= \statethreshold \wedge q_p(k-1)=1).
	\end{cases}
\end{equation*}

We are now ready to describe the opinion dynamics model that we consider in this work. This dynamics adapts the CODA model in \cite{NilCoSa2016} to include the external dynamics of $\statet$:

\begin{equation}\label{eq:opinion_dynamic}
	\begin{split}
		\opiniontplusun{i} ={}&  \opiniont{i} + \left(1 - \opiniont{i}^2 \right) \big[\beta\left( \actiont{p} - \opiniont{i} \right)\\
		&+ \left(1-\beta\right)\frac{1}{n_{i}} \sum_{j \in \SetN{i}}\left( \actiont{j} - \opiniont{i} \right) \big],
	\end{split}
\end{equation}

where $0 \leq \beta \leq 1$ encapsulates the trade-off between the environment state observed through $\actiont{p}$ and the opinions of the neighbors. 
We note that the complete model coupling CODA and the external dynamics is described by \eqref{eq:emissions}-\eqref{eq:opinion_dynamic}.

We emphasize a natural partition of $\Vcal$ in two subsets $\SetNst{}{-} = \left\{i \in \Vcal \mid \actiont{i} = -1\right\}$ and $\SetNst{}{+} = \left\{i \in \Vcal \mid \actiont{i} = 1\right\}$. In the following, we denote by $n^{-}(k)$ and $n^{+}(k)$ the cardinality of $\SetNst{}{-}$ and $\SetNst{}{+}$, respectively. Similarly, for an agent $i$ we denote by $\SetNst{i}{-} = \SetN{i} \cap \SetNst{}{-}$ and $\SetNst{i}{+} = \SetN{i} \cap \SetNst{}{+}$ and by $\npos{i}$ and $\nneg{i}$ the cardinalities of these sets. 

\section{Analysis of the model}\label{sec:analysis}

Before starting the analysis of the model introduced in the previous section, let us observe that extreme opinion values $\opinion{i}(0)\in\{-1,1\}$ do not evolve in time. We also observe that the definitions of $q$ and $q_p$ are rigorous only if $\opinion{i}(0)\neq 0, \forall i\in\Vcal$ and $p(0)\neq\statethreshold$. Therefore, the following assumption is perfectly justified by our setup.  
\begin{assumption}\label{assumption:not_on_border}
	For all $i \in \Vcal$, $\opinion{i}(0) \in \left( -1, 1\right)\setminus\{0\}$ and $p(0)\neq\statethreshold$.  
\end{assumption}

\subsection{Characterization of opinion equilibria}

In the following, we analyze the asymptotic behavior of opinions that follows the dynamics \eqref{eq:opinion_dynamic}. In other words, we assume that the external signal has an exogenous decoupled evolution. 

To simplify our further reasoning, we introduce the following notation 
\begin{equation}\label{eq:f_i}
	f_i(k) =  \left(1 - \beta\right) \frac{\npost{i} - \nnegt{i}}{n_i} + \beta \actiont{p}.
\end{equation}

\begin{lemma}\label{lemma:monotonicity}
	Let $i \in \Vcal$, $\opinion{i}(0) \in \left(-1, 1\right)$. Then for all $k \in \N$, one of the following relation holds
		\begin{equation}\label{eq:lower_inequality}
		\opiniont{i} < \opiniontplusun{i} < f_i(k),
	\end{equation}
	\begin{equation}\label{eq:upper_inequality}
		\opiniont{i} > \opiniontplusun{i} > f_i(k),
	\end{equation}
	or, 
	\begin{equation}\label{eq:equality}
		\opiniont{i} = \opiniontplusun{i} = f_i(k).
	\end{equation}

\end{lemma}

\begin{proof}
	Let us first observe that  $\sum_{j \in \SetN{i}}^{} \actiont{j} = \npost{i} - \nnegt{i} = 2\npost{i} - n_{i}$, since $\npost{i} + \nnegt{i} = n_{i}$ for all $k\in \N$. 
 Then, using \eqref{eq:f_i}, one rewrites \eqref{eq:opinion_dynamic} as:
	\begin{align}
			\opiniontplusun{i} ={}& \opiniont{i} + \left(1 - \opiniont{i}^2\right) \big[ \beta \left(\actiont{p} - \opiniont{i}\right)\nonumber\\
							& {} + \left(1 - \beta\right) \frac{ 1}{n_{i}}\left(\npost{i} - \nnegt{i} - n_{i}\opiniont{i}\right)  \big]\nonumber\\
							={} & \opiniont{i} + \left(1 - \opiniont{i}^2\right) \big[ f_i(k) - \opiniont{i}   \big] \label{eq:opinion_dynamic_rewritte}
	\end{align}
	We continue our reasoning by induction. From equation \eqref{eq:opinion_dynamic_rewritte} it is straightforward that if $\opiniont{t} < f_i(k)$ then $\opiniont{i} < \opiniontplusun{i}$ and $\opiniontplusun{i}<f_i(k)$. Reversely, if $\opiniont{i} > f_i(k)$ then $\opiniontplusun{i} > \opiniont{i}$ and $\opiniontplusun{i}>f_i(k)$. Finally, if $\opiniont{i} = f_i(k)$ then $\opiniontplusun{i} = f_i(k)$.

\end{proof}

\begin{proposition}
	Let $i\in \Vcal$ and assume that Assumption \ref{assumption:not_on_border} holds. If $\left(\actiont{p}\right)_{k\geq 0}$ and $\left(\npost{i}\right)_{k\geq 0}$ are stationary sequences with limit $\action{p}^*$ and $\npos{i}{}^*$, respectively. Then the sequence of opinion $\left(\opiniont{i}\right)_{k\geq 0}$ converges to 
	\begin{equation*}
		\opinion{i}^* = \lim_{k \to \infty} \opiniont{i} = \left(1 - \beta\right) \frac{2 \npos{i}{}^* - n_{i}}{n_{i}} +  \beta\action{p}^* \in \Q + \beta \Q
	\end{equation*}
\end{proposition}

\begin{proof}

Let us note $f_i^* := \left(1 - \beta\right) \left(2 \npos{i}{}^* - n_{i}\right)/n_i + \beta \action{p}^* \in \left[-1, 1\right]$ and remarks that $\npos{i}{}^*$ and $\action{p}^* \in \N$ since those are stationary sequences in $\N$. Then $f_i^* \in \Q + \beta \Q$.
With Assumption \ref{assumption:not_on_border}, we have $\opinion{i}(0) \in \left(-1, 1\right)$, so that Lemma \ref{lemma:monotonicity} applies. Then by induction 
\begin{equation*}
	\forall k \geq k^*, \quad \opiniont{i} \leq \opiniontplusun{i} \leq f_i^*\leq 1 
\end{equation*}
or
\begin{equation*}
	\forall k \geq k^*, \quad \opiniont{i} \geq \opiniontplusun{i} \geq f_i^* \geq 1.
\end{equation*}
Then since $\left(\opiniont{i}\right)_{k\geq 0}$ is a bounded monotonous sequence, it converges. Let denote $\opinion{i}^*$ that limit. Now from above inequalities, if $\opinion{i}^* = 1 $ then $f_i^* = 1$. Conversely, if $\opinion{i}^* = -1$, we have that $f_i^* = -1$. Finally, if $\opinion{i}^* \in \left(-1, 1\right)$ then \eqref{eq:opinion_dynamic_rewritte} rewrite as 
\begin{equation*}
	\frac{\opiniontplusun{i} - \opiniont{i}}{\left(1 - \opiniont{i}^2\right)} + \opiniont{i} = f_i^*.
\end{equation*}
Taking the limit of the previous cancel the first term of the left-hand side and we have that $\opinion{i}^* = f_i^* $.
\end{proof}

We will later see that opinions can either have an oscillatory chaotic behavior or they converge to a limit cycle. For a discrete-time system given by \(x(k+1) = H(x(k))\), if there is a natural number \(m > 1\) for which there exist \(m\) successive convergent sub-sequences \(x_0(k) = x(mk)\), \(x_1(k)=x(mk +1)\), ..., \(x_{m-1}(k) = x(m(k+1) - 1)\), then the overall sequence converges to a limit cycle of length $m$ defined by the limits of the $m$ sub-sequences. 

\begin{corollary}
	Under Assumption \ref{assumption:not_on_border}, if $\left(\actiont{p}\right)_{k\geq 0}$ is a stationary sequence and $\left(\npost{i}\right)_{k\geq 0}$ converges to a limit cycle of length $m$, then $\opinion{i}$ also converges to a limit cycle of length $m$ denoted $\bar{\opinion{i}} \subset \Q + \beta \Q +\dots + \beta^m \Q$. 
\end{corollary}

\subsection{Asymptotic behavior of the external/pollution dynamics}

As pointed out in the previous subsection, the stationarity of $\left(\actiont{p}\right)_{k\geq 0}$ plays a major role in the asymptotic behavior of the opinions $\opinionv$. Consequently, in this subsection we provide a sufficient condition ensuring that $\left(\actiont{p}\right)_{k\geq 0}$ is a stationary sequence. In this setting, the opinions behave as in the CODA model provided in \cite{NilCoSa2016} since the players are influenced by the external dynamics uniformly with respect to time after the sequences become stationary. We can rewrite the dynamics \eqref{eq:state_dynamic} by injecting \eqref{eq:emissions}
\begin{equation}\label{dyn_p}
	\statetplusun = \gamma \statet + \npost{}\emax +\nnegt{}\emin
\end{equation} 

Therefore, the pollution $p$ reaches an equilibrium only if $(\npost{})_{k\ge0}$ (and implicitly $(\nnegt{})_{k\ge0}$) is stationary. Let us suppose that 
\[
\lim_{k\to\infty}\npost{}=n^+, \quad \lim_{k\to\infty}\nnegt{}=n^-
\]In this case the equilibrium $p^*$ is given by
\begin{equation}\label{eq:equilibria_state}
	p^* = \frac{n^+\emax +n^-\emin}{1 - \gamma},
\end{equation}
which now only depends on the partition of actions of the individuals in the social network. 

In order to guarantee that $\left(\actiont{p}\right)_{k\geq 0}$ is stationary one needs to ensure that $\left(\sign{\statet - \statethreshold}\right)_{k \geq 0}$ is stationary. In other words, for $k$ sufficiently large the value of $\statet$ does not cross the threshold $\statethreshold$.   
\begin{lemma}
The sequence $\left(\actiont{p}\right)_{k\geq 0}$ is stationary for any graph $\Gcal$ with $N$ individuals if there exists $k$ such that either $\Big(\statet\le\frac{N\emax}{1-\gamma} \wedge \frac{N\emax}{1-\gamma}\le\statethreshold\Big)$ or
$\Big(\statet\ge\frac{N\emin}{1-\gamma} \wedge \frac{N\emax}{1-\gamma}\ge\statethreshold\Big)$.
\end{lemma}

\subsection{Preservation of action}
In the following, we investigate under which condition we have that $\actiont{i} = \actiontplusun{i}$. The following result is instrumental for our purposes.

\begin{lemma}\label{lemma:stab_action}
	Let $i\in \Vcal$, then the following statements hold true:
 \begin{enumerate}
\item if $f_i(k) \geq 0 $ and $\actiont{i} = 1$ then $\actiontplusun{i}=1$,\\			
\item if $f_i(k) \le 0 $ and $\actiont{i} = -1$ then $\actiontplusun{i}=-1$.
\end{enumerate}
\end{lemma}

\begin{proof}
As proven in Lemma \ref{lemma:monotonicity} one of \eqref{eq:lower_inequality}, \eqref{eq:upper_inequality} or \eqref{eq:equality} holds true.\\
1) If \eqref{eq:lower_inequality} is verified one has $\opiniontplusun{i} > \opiniont{i}$ meaning that $\actiont{i} = 1$ implies $\actiontplusun{i}=1$.\\
If \eqref{eq:upper_inequality} or \eqref{eq:equality} holds, then $\opiniont{i} \geq f_i(k)$. Consequently, $f_i(k) \geq 0$ ensures $\opiniontplusun{i} \geq 0$ or equivalently $\actiontplusun{i}=1$.\\ 
2) If \eqref{eq:lower_inequality} or \eqref{eq:equality} holds,  then $\opiniont{i} \leq f_i(k)$. Therefore, if $f_i(k) < 0$ yields $\opiniontplusun{i}<0\Leftrightarrow \actiontplusun{i}=-1$.\\
If \eqref{eq:upper_inequality} holds, one has that $\opiniontplusun{i} < \opiniont{i}$ meaning that $\actiont{i} = -1$ implies $\actiontplusun{i}=-1$.
\end{proof}

When $\beta <1/(1 + n_i)$ Lemma \ref{lemma:stab_action} can be refined as follows. 
\begin{lemma}
	Let $i\in \Vcal$ and assume that $\beta <1/(1 + n_i)$. The following statements hold:\\
	$-$ if $\npost{i} > \nnegt{i} $ and $\actiont{i} = 1$ then, $\actiontplusun{i}=1$,\\			
	$-$ if $\nnegt{i} < \npost{i}$ and $\actiont{i} = -1$ then, $\actiontplusun{i}=-1$.
\end{lemma}

\begin{proof}
Notice that $\beta < 1 / (1+n_i)$ implies $\displaystyle\frac{\beta}{1-\beta} < \displaystyle\frac{1}{n_i}$. Recalling that $|\actiont{p}|=1$ one obtains that $\left| \displaystyle\frac{n_i \actiont{p}\beta}{1 - \beta}\right|<1$. Notice also that  
	\begin{align*}
		f_i(k) &= \left(1 - \beta\right) \frac{\npost{i} - \nnegt{i}}{n_i} + \beta \actiont{p} \\
		& =\frac{1-\beta}{n_i} \left(\npost{i} - \nnegt{i} +\frac{n_i \actiont{p}\beta}{1 - \beta}\right) 
	\end{align*}
Therefore 
\[
\sign {f_i(k)}=\sign{\npost{i} - \nnegt{i} +\frac{n_i \actiont{p}\beta}{1 - \beta}},
\] and taking into account that $\npost{i}, \nnegt{i}\in\N$ the desired result yields from Lemma \ref{lemma:stab_action}.
\end{proof}

Throughout the paper, we denote by $|A|$ the cardinality of a set $A$. We provide a definition for some cluster in the graph such that the opinion will not change through time 

\begin{definition}
	We say that a subset of agents $A \subset \Vcal$ is a weakly robust polarized cluster if the following hold:
	\begin{itemize}
		\item $\forall i, j \in A$, $\action{i}(0) = \action{j}(0)$,
		\item $\forall i \in A$, $| \SetN{i} \cap A | \geq | \SetN{i} \setminus A |-  \beta / (1-\beta)|\SetN{i}|$. 
	\end{itemize}
\end{definition}

\begin{proposition}
    \label{proposition:weakly}
	If $A$ is a weakly robust polarized cluster 
 and $\actiont{p} = \action{i}(0)$ for all $k\in \N$ and $i\in A$, then
$$\actiont{i} = \action{i}(0), \ \forall i \in A, \forall k \in \N.$$ 
\end{proposition}

\begin{proof}
	The proof will be done by induction. Let us suppose that $\forall i \in A$ one has $\action{i}(0) = 1$. Assume that for a given  $k' \in \N$ one has $\action{i}(k') = 1, \ \forall i \in A$. Since the interaction graph is fixed and $A$ is a weakly robust polarized cluster the following holds true $| \SetN{i} \cap A | \geq | \SetN{i} \setminus A | -  \beta / (1-\beta)|\SetN{i}|$. Noticing that $\SetN{i} \cap A \subseteq \SetNs{i}{+}(k')$ one obtains that $\npos{i}(k') \geq \nneg{i}(k') - n_i \beta / (1 - \beta)$ which is equivalent to $f(k') \geq 0$. Applying Lemma \ref{lemma:stab_action} one gets $\action{i}(k'+1) = 1$ and the induction is complete. Similar reasoning applies when $\action{i}(0) = -1,\ \forall i \in A$.
\end{proof}

The preservation of action in a weakly polarized cluster is subject to a constant value of $q_p$ over the time. In order to get rid of this constraint, we introduce the following concept.
\begin{definition}
	We say that a subset of agents $A \subset \Vcal$ is a strongly robust polarized cluster if the following hold:
	\begin{itemize}
		\item $\forall i, j \in A$, $\action{i}(0) = \action{j}(0)$,
		\item $\forall i \in A$, $| \SetN{i} \cap A | \geq | \SetN{i} \setminus A |+  \beta / (1-\beta)|\SetN{i}|$. 
	\end{itemize}
\end{definition}

\begin{proposition}
	If $A$ is a robust polarized cluster 
 then
  $\forall i \in A$, $\forall k \in \N$, $\actiont{i} = \action{i}(0)$. 
\end{proposition}
\begin{proof}
    Apply proof of Proposition \ref{proposition:weakly} while considering \({| \SetN{i} \cap A | \geq | \SetN{i} \setminus A |+  \beta / (1-\beta)|\SetN{i}|}\).
\end{proof}

It is worth noting that for $\beta > 1/2$ we cannot have robust polarized cluster since $\beta/(1 - \beta) >1$ and  $| \SetN{i} \cap A | \leq  \beta / (1-\beta)|\SetN{i}|$ for any $A \subset \Vcal$. This fact will have importance in the following.

\section{Analysis of the synchronized behavior}
\label{sec:synchro_behav}
In the case where $\beta >1/2$ the state of the system does not converge towards a steady state. Instead, one has oscillations that may be either chaotic or converging to a limit cycle. This is illustrated in Fig. \ref{fig:bifurcationdiagrambetaopinionwithconsensus} where we can see that for $\beta$ close to one we get a limit cycle while for $\beta>1/2$, in a wide range, one has a chaotic behavior. For the sake of simplicity, we assume in the following that all the opinions are synchronized. 
\begin{figure}
	\centering
	\includegraphics[width=0.9\linewidth]{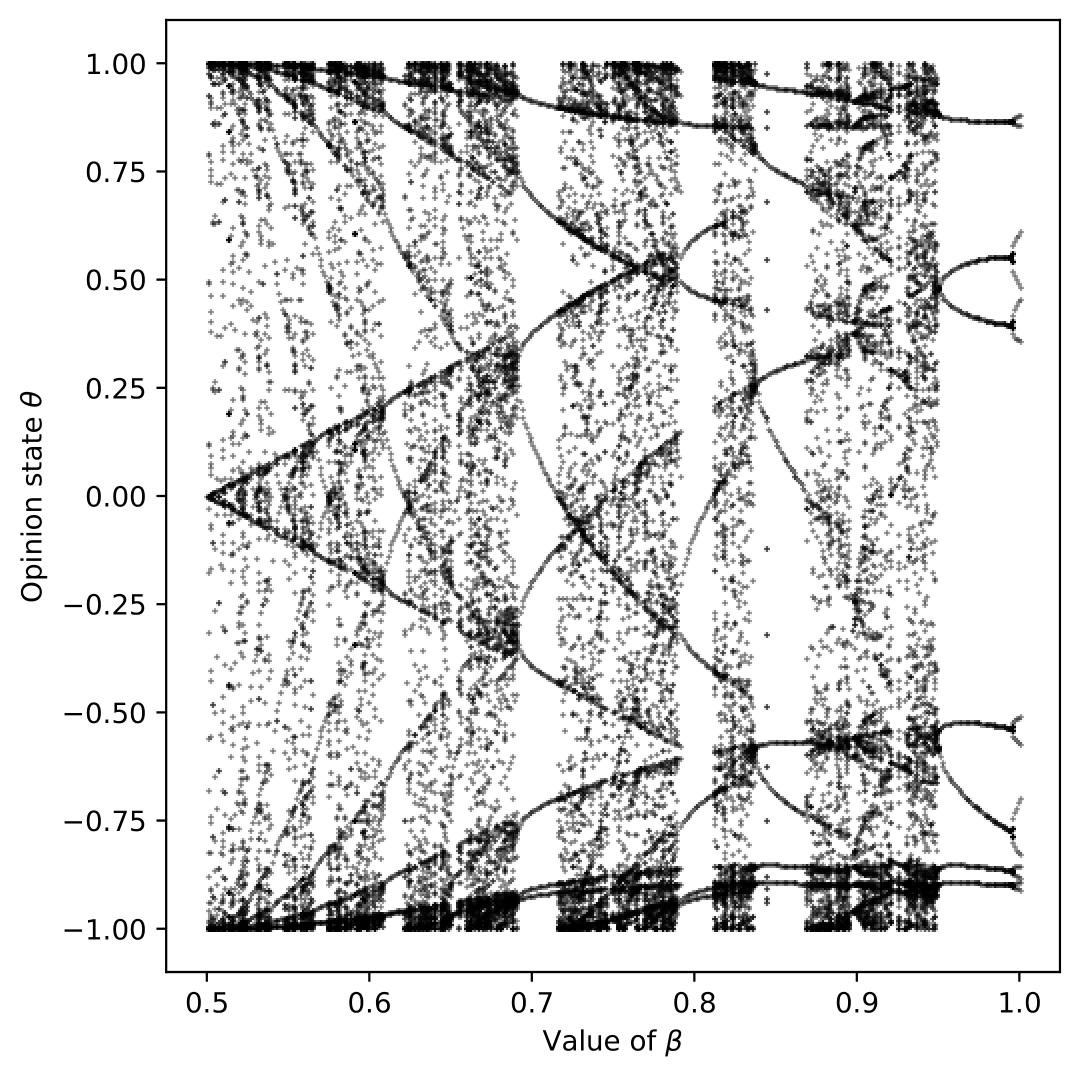}
	\caption{Bifurcation diagram of the opinion for $0.5<\beta<1$. $N=20$, $\opinion{}(0)= 0.4$, $\state(0)=100$, $\statethreshold = 15$, $\emin = 0$, $\emax =1$ and $\gamma = 0.5$,}
	\label{fig:bifurcationdiagrambetaopinionwithconsensus}
\end{figure}

\begin{definition}
	We say that an opinion state \(\boldsymbol{\theta} = \left( \theta_1 \cdots \theta_N \right)^\top\) is Fully Synchronized (FS) if $\forall i, j\in\Vcal$, $\theta_i = \theta_j$.
\end{definition}

When the opinion state is FS, the opinion of an agent $i\in \Vcal$ is equal to the opinion of any other agent in $\Vcal$. Therefore, in the remainder of this section, we will denote $\theta(k)$ and $q(k)$ the common opinion and action of all the agents at time $k$. In other words, we omit the agent index when referring to its opinion or action.

\begin{proposition}
    The FS property is forward invariant over time, i.e. if \opistate{k} is FS at time $k\in \N$, then \opistate{k+1} is FS.
\end{proposition}

\begin{remark}
	If \opistate{k} is FS then $f(k) = (1-\beta) \actiont{} + \beta \actiont{p}$. 
\end{remark}
Indeed, under the assumption of FS, one has that for all $i \in \Vcal$, $(\npost{i} - \nnegt{i})/n_i = \actiont{i}=\actiont{}$.

In the FS regime, the action space reduces to  
\begin{align*}
\mathcal{S} &= \left\{ (\action{}, \action{p}) \mid \action{} \in \{ -1, 1\}, \action{p} \in \{-1, 1 \}\right\} \\
&= \left\{ (-1,-1), (-1,1), (1,-1), (1,1)\right\}.
\end{align*}

Notice that each point in $\mathcal{S}$ corresponds to a partition of the state space in four sets. For instance $(1,1)$ corresponds to $\{\opiniont{i}\ge 0,\forall k\state\le \statethreshold\}$ In order to prove the oscillatory behavior of the system \eqref{eq:state_dynamic}-\eqref{eq:opinion_dynamic} we show that in general $\mathcal{S}$ does not contain equilibrium points. This means that the trajectory of the system cannot remain in a certain partition which means it cannot converge towards a steady state.
\begin{proposition}\label{proposition:not_equilibrium}
	Assume that $\beta > 1/2$ and the opinion state is FS at time $k\in \N$. Then the points $\left\{(1,-1), (1,-1)\right\}$ are not equilibrium in the action space.
\end{proposition}

\begin{proof}
	We proceed by contradiction. The reasoning is similar for each of the two points so we will focus on the first one. Let us assume that $(1,-1)$ is an equilibrium point. This means that if there exists $k^*\in \N$ such that $\action{}(k^*)=1,\ \action{p}(k^*)=-1$ then for any $k> k^*$ we have $\actiont{} = 1$ and $\actiont{p} = -1$. We notice that in this case, one has $f(k) = 1-2\beta, \ \forall k> k^*$. Then the dynamics \eqref{eq:opinion_dynamic_rewritte} becomes
	\begin{align*}
		\opiniontplusun{} = \opiniont{} + \left(1 - \opiniont{}^2\right) \left(1 -2\beta - \opiniont{}\right), \ \forall k> k^*.
	\end{align*}
	Recalling that $\actiont{}=1,\ \forall k> k^*$ one deduces that $\opiniont{}\geq0, \ \forall k> k^*$. Consequently, $\forall k> k^*$ one obtains
	\begin{align*}
		\opiniontplusun{} - \opiniont{} = \left(1 - \opiniont{}^2\right) \left(1 -2\beta - \opiniont{}\right) \leq 1-2\beta.
	\end{align*}
	Iterating the inequality above for consecutive values of $k$ it results that for any $n\in\N$ the following holds: \begin{equation}
 \label{iteration}
	 \opinion{}(k+n) \leq n(1-2\beta) + \opiniont{}\leq n(1-2\beta) +1.   
	\end{equation}
 Let us recall that $1-2\beta <0$. Therefore, if we consider in \eqref{iteration} a sufficiently large $n$ (i.e. $n> 1/(2\beta-1)$) we get $\opinion{}(k+n) <0$ yielding$\action{}(k+n) = -1$ which contradicts the assumption that $(1,-1)$ is an equilibrium in the action space. 
\end{proof}

\begin{proposition}
	Assume $\beta > 1/2$ and \(\boldsymbol{\theta}\) being FS. The following statements hold true.
	\begin{itemize}
		\item $(1,1)$ is an equilibrium in the action space if and only if $\state_{\max} \le \statethreshold$ where $\state_{\max} = N\emax/(1 - \gamma)$. In this case $(\opinionv,\state)=(\mathbf{1}_n,\state_{\max})$ is a stable equilibrium of the coupled dynamics \eqref{eq:state_dynamic}-\eqref{eq:opinion_dynamic}.
		\item $(-1,-1)$ is an equilibrium in the action space if and only if $\state_{\min} \ge \statethreshold$ where $\state_{\min} = N \emin/(1 - \gamma)$. In this case $(\opinionv,\state)=(-\mathbf{1}_n,\state_{\min})$ is a stable equilibrium of the coupled dynamics \eqref{eq:state_dynamic}-\eqref{eq:opinion_dynamic}.
	\end{itemize}
\end{proposition}

\begin{proof}
The reasoning is symmetric and it is sufficient to focus only on the first statement.\\
"$\Rightarrow$" We assume that $(1,1)$ is an equilibrium and show that $\state_{\max} < \statethreshold$.\\
Suppose that there exists $k^*\in \N$ such that $\action{}(k^*)=1,\ \action{p}(k^*)=1$. Then, for all $k > k^*$, we have $\actiont{} = 1$ 
and $\actiont{p} = 1$. 
Therefore, for all $k > k^*$ one has $\statet \le \statethreshold$ and the dynamics \eqref{eq:state_dynamic} becomes
	\begin{equation}\label{particular state dyn}
		\statetplusun = \gamma \statet + N \emax.
	\end{equation}
Since $\gamma\in(0,1)$ the dynamics above is asymptotically stable and $\statet$ converges to the equilibrium defined in \eqref{eq:equilibria_state} with $n^+=N,\ n^-=0$ i.e. $\state^*=\state_{\max} = N  \emax/(1 - \gamma)$. We conclude that $\state_{\max} \le\statethreshold$.
	
"$\Leftarrow$"	We assume that $\state_{\max} \le \statethreshold$ and show that $(1,1)$ is an equilibrium.\\
One proves that if $k$ is such that $\actiont{}=1,\ \action{p}(k)=1$ then $\actiontplusun{} =1,\ \action{p}(k+1)=1$.
Notice first that in the case under consideration $f(k)=1$. Therefore, dynamics \eqref{eq:opinion_dynamic} becomes
\begin{equation}\label{13}
    \opiniontplusun{} =\opiniont{}+(1-\opiniont{} ^2)(1-\opiniont{})\ge\opiniont{}.
\end{equation}
Since $\actiont{}=1$ one deduces from \eqref{13} that $\actiontplusun{}=1$ as well.\\
On the other hand $\action{p}(k)=1$ implies $\statet{}<\statethreshold$	and $\state_{\max} \le \statethreshold$ is equivalent to $N\emax\le(1-\gamma)\statethreshold$. Therefore, \eqref{particular state dyn} becomes
\[
\statetplusun<\gamma\statethreshold+(1-\gamma)\statethreshold=\statethreshold
\] yielding $\actiontplusun{}=1$.
By recursive reasoning one gets that $(1,1)$ is an equilibrium point.

We already noticed that in the case under study $\state_{\max}$ is an asymptotically stable point for \eqref{eq:state_dynamic} that takes the particular form \eqref{particular state dyn}. On the other hand the dynamics \eqref{eq:opinion_dynamic} simplifies as \eqref{13} whose stable equilibrium is $\mathbf{1}_n$.
\end{proof}

In the general case when $\state_{min}<\statethreshold<\state_{\max}$  neither $(1,1)$ nor $(-1,-1)$ is an equilibrium in the action space. Since no equilibrium exists in this case, the trajectory of \eqref{eq:state_dynamic}-\eqref{eq:opinion_dynamic} will switch an infinite number of times between the four sets of the partition defined by the action space $\mathcal{S}$. 

\section{Numerical results}\label{sec:num_results}

First, we will illustrate the same kind of result as in \cite{NilCoSa2016} over a square lattice. Finally, we will present the different behaviors we can observe for a complete graph with FS property. 
\subsection{Square lattice}
\begin{figure}
    \centering
    \includegraphics[width=0.8\linewidth]{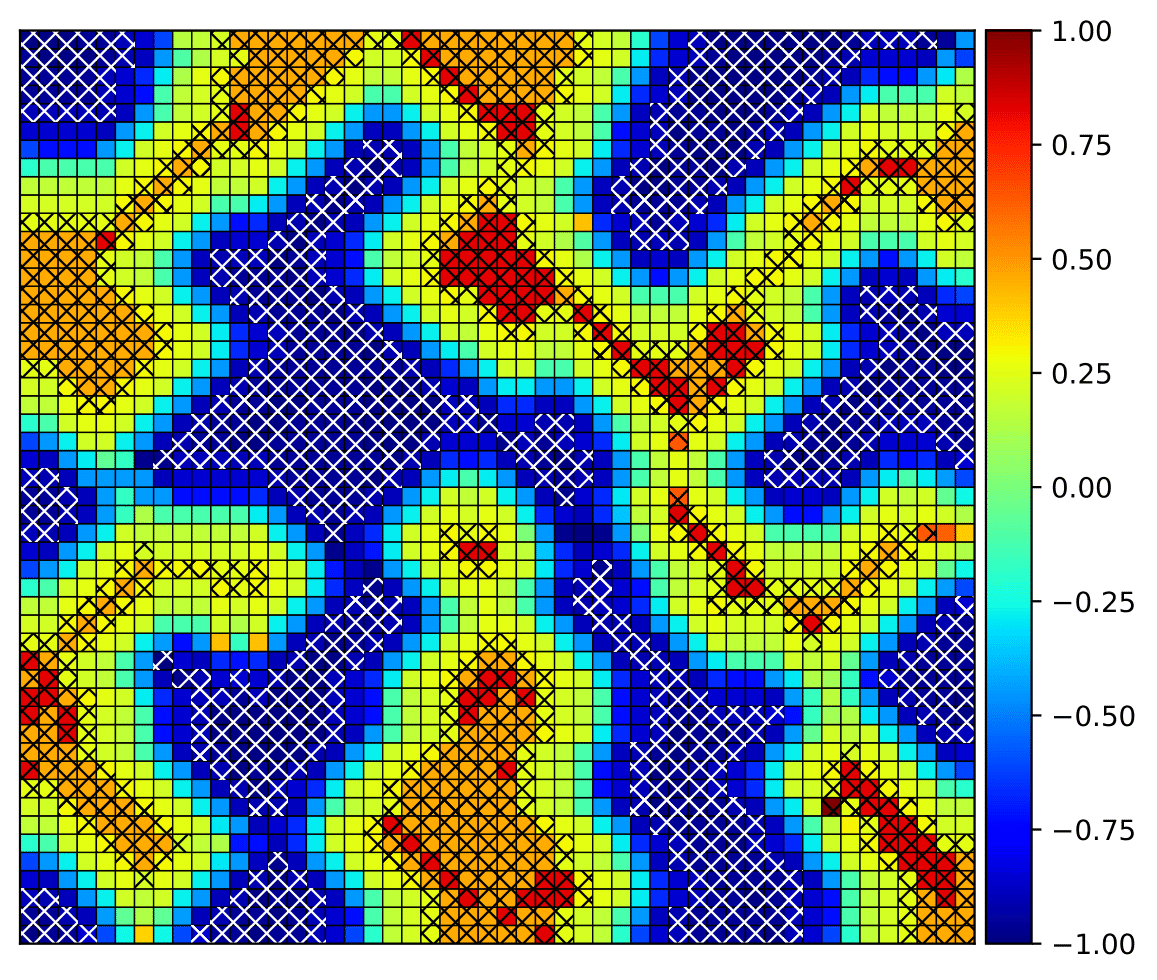}
    \caption{Visualization of Opinion Dynamics on a 50 × 50 Square Lattice for $\beta = 0.45$: Initial opinions are randomly distributed as i.i.d. uniform variables between -1 and 1, with the resultant opinions after 100 iterations represented by each colored square cell. Agents engage in communication with their adjacent cells (above, below, left, and right). The cells marked with crosses indicate the presence of strongly robust polarized clusters; black crosses correspond to action -1, and white crosses denote action 1.}
    \label{fig:beta_inf_05}
\end{figure}

\begin{figure}
    \centering
    \includegraphics[width=\linewidth]{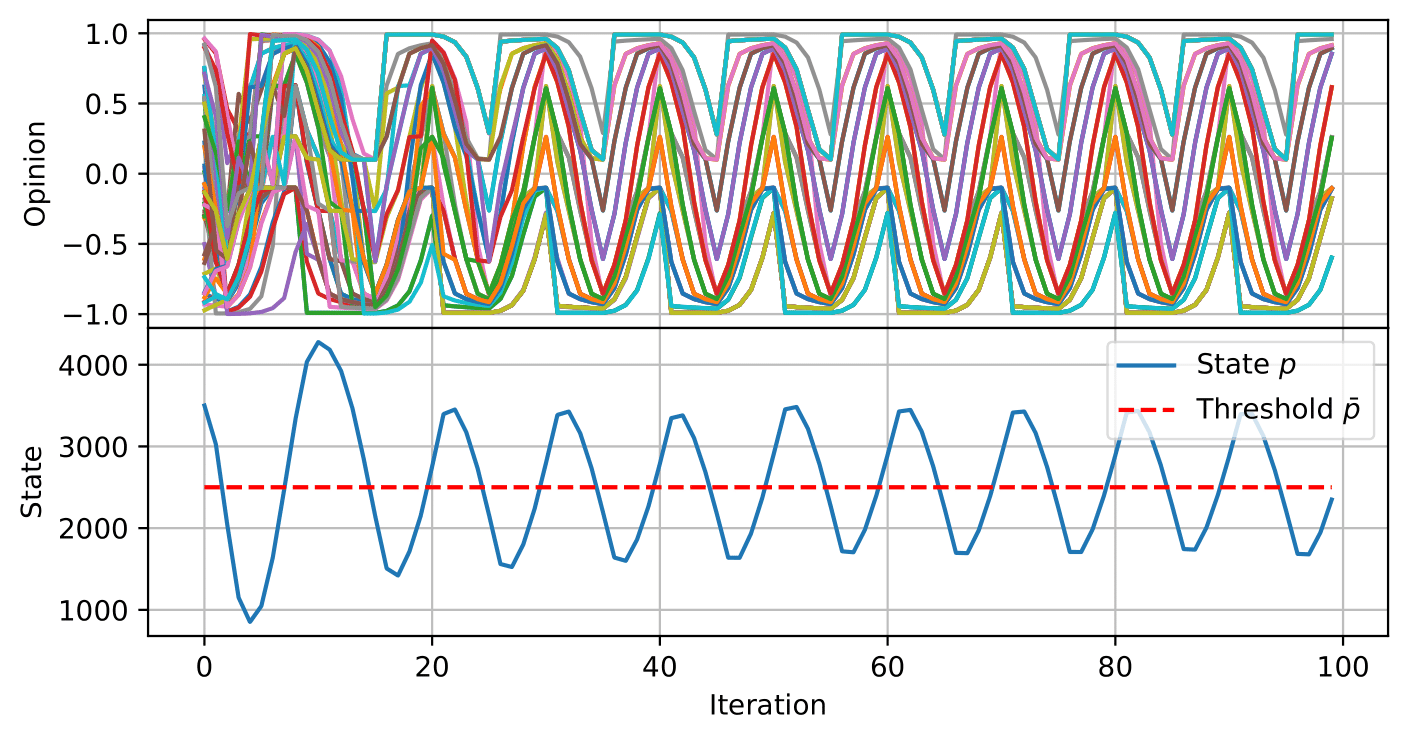}
    \caption{Depiction of Dynamical Evolution in the 50 × 50 Square Lattice from Figure \ref{fig:beta_inf_05}: The upper panel showcases the trajectory of each agent's opinion over iterations, while the lower panel illustrates the corresponding state evolution. The red dashed line marks the state threshold.}
    \label{fig:beta_inf_05_state}
\end{figure}

\begin{figure*}[ht]
    \centering
     
    \includegraphics[width=\linewidth]{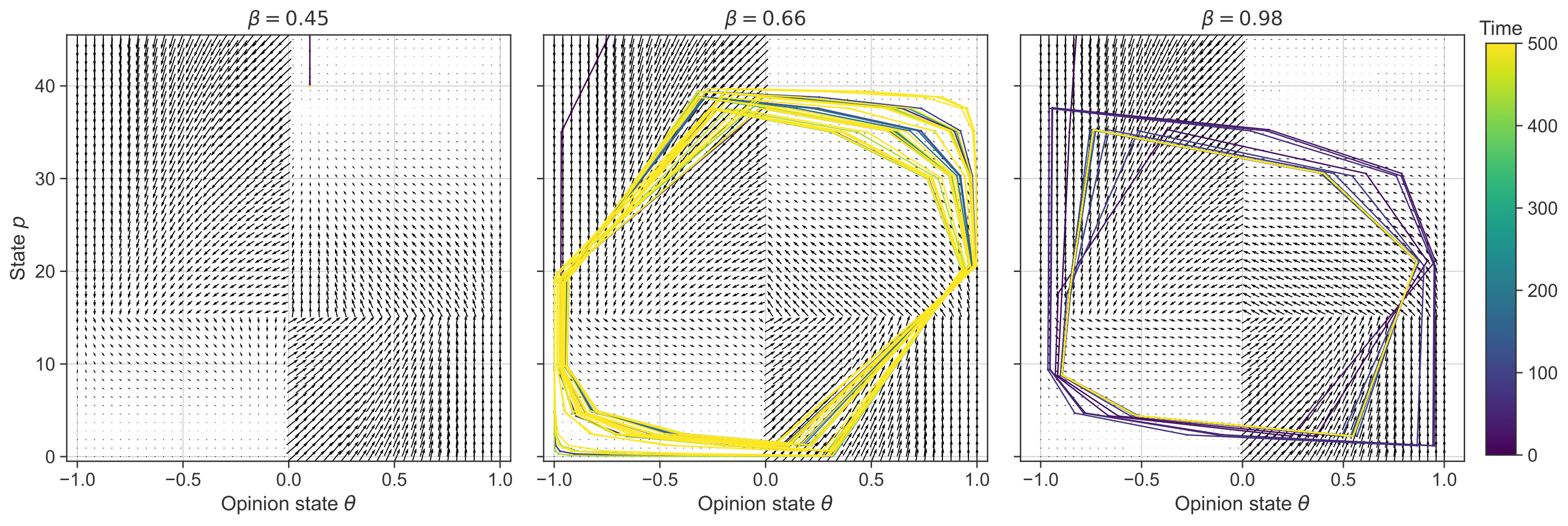}

    \caption{Trajectory dynamics. Left: two possible equilibria. Center: no equilibrium or limit cycle. Right: limit cycle.}
    \label{fig:asymptotic_cases}
\end{figure*}
Our study visualizes results with interactions based on a square lattice topology. In Figure \ref{fig:beta_inf_05}, we note the persistence of resilient clusters even after numerous iterations. The evolution of opinions and the corresponding state through iterations is depicted in Figure \ref{fig:beta_inf_05_state}. We can observe that opinions and state converge fast to a limit cycle. On Figure \ref{fig:beta_inf_05}, we see that the opinions are polarized on the graph. There are many robust clusters for both action 1 and -1. Each of them is separated by a frontier that seem to have the same length between the clusters of opposite action and them. We can see on Figure \ref{fig:beta_inf_05_state} that there is no agent with a constant opinion, the opinions are on a limit cycle. The ones that stay
with a constant action form robust clusters, as illustrated in Figure \ref{fig:beta_inf_05}.

\subsection{Complete graph with FS}

In Figures \ref{fig:asymptotic_cases}, we identify three unique behaviors displayed by the dynamical systems: stable equilibrium (when $\beta<0.5$), chaotic patterns, and a collection of limit cycles. These simulations were conducted on a complete graph comprising 20 nodes, all of which possess the FS property at initial time. The starting opinion is set at $\opinion{}(0) = 0.4$, and the initial state is taken at $\state(0)=100$, with a threshold of $\statethreshold = 15$. For the emission dynamics, the range is between $\emin = 0$ and $\emax = 1$, and the decay rate is defined by $\gamma = 0.5$. We represent the discrete trajectory linking each consecutive couple $(\opiniont{}, \statet)$ by a straight line. The color of the line represents the time when this shift occurs. Moreover, we present the subsequent vector field of the dynamics illustrated by the quivers. The speed of the dynamics is given by the length of the arrows. 

A quick observation reveals that, based on the given parameters, the dynamics of the opinion and state undergo significant variations in response to the value of $\beta$. For instances when $\beta <0.5$, both the opinion and state quickly stabilize at an equilibrium. However, as depicted in Figure \ref{fig:bifurcationdiagrambetaopinionwithconsensus}, given the same conditions, when $\beta >0.5$ we discern two potential behaviors. The first sees $\beta$ positioned outside all limit cycle intervals, resulting in chaotic behavior as illustrated in the center of Figure \ref{fig:asymptotic_cases}. The second showcases a limit cycle, as depicted in the right of Figure \ref{fig:asymptotic_cases}. Specifically, in the Chaos case of Figure \ref{fig:asymptotic_cases}, the sequence $\left( (\opiniont{}, \statet) \right)_{k \geq 0 }$ never reiterates, thus forming its unique pattern. Conversely, when $\beta$ falls within a limit cycle interval, the sequence $\left( (\opiniont{}, \statet) \right)_{k \geq 0 }$ converge towards a limit cycle.

\section{Conclusion}\label{sec:ccl}
In this paper, we have introduced and analyzed a CODA model coupled with an external dynamic state. Specifically, we consider that the external dynamics represent a very simple pollution model in which the emission level depends on the actions of the individuals in the social network. Conversely, the opinions are both influenced by the actions of the neighbors and the pollution level (above or below a given threshold). We have shown that different behaviors are possible, ranging from convergence to a steady state to chaotic behavior of the coupled dynamics. Numerical examples illustrate our theoretical results. 

\bibliographystyle{IEEEtran}
\bibliography{CODApaper}

\end{document}

%% file: packages.tex
\usepackage{amsfonts}
\usepackage{dsfont}
\usepackage{amsmath}
\usepackage{amssymb}
\usepackage{amsthm}
\usepackage{bm}
\usepackage{subcaption}
\usepackage{comment}

\usepackage{xcolor}

\newtheorem{definition}{Definition}
\newtheorem{proposition}{Proposition}

\newtheorem{lemma}{Lemma}
\newtheorem{corollary}{Corollary}
\newtheorem{assumption}{Standing Assumption}
\newtheorem{remark}{Remark}

%% file: basic_macros.tex
\newcommand{\mc}{\mathcal}

%% file: notation.tex
\newcommand{\Vcal}{\mc{V}}
\newcommand{\Gcal}{\mc{G}}
\newcommand{\Ecal}{\mc{E}}

\newcommand{\SetN}[1]{\mc{N}_{#1}}
\newcommand{\SetAction}{\left\{ -1, 1\right\}}

\newcommand{\opinion}[1]{\theta_{#1}}
\newcommand{\opinionv}{\bm{\theta}}

\newcommand{\opiniont}[1]{\theta_{#1}(k)}

\newcommand{\opiniontplusun}[1]{\theta_{#1}(k+1)}

\newcommand{\action}[1]{q_{#1}}

\newcommand{\actiont}[1]{q_{#1}(k)}

\newcommand{\actiontplusun}[1]{q_{#1}(k+1)}

\newcommand{\state}{p}
\newcommand{\statet}{\state(k)}
\newcommand{\statetplusun}{\state(k+1)}

\newcommand{\statethreshold}{\bar{\state}}

\newcommand{\emissiont}[1]{e_{#1}(k)}

\newcommand{\emin}{e_\text{min}}
\newcommand{\emax}{e_\text{max}}

\newcommand{\SetNs}[2]{\SetN{#1}^{#2}}
\newcommand{\SetNst}[2]{\SetN{#1}^{#2}(k) }

\newcommand{\npos}[1]{n_{#1}^{+}}
\newcommand{\npost}[1]{n_{#1}^{+}(k)}

\newcommand{\nneg}[1]{n_{#1}^{-}}
\newcommand{\nnegt}[1]{n_{#1}^{-}(k)}

\newcommand{\discretset}[2]{ \left\{#1, \dots, #2 \right\}}

\newcommand{\Vset}{\discretset{1}{N}}

\newcommand{\sign}[1]{\text{sgn}\left(#1\right)}

\newcommand{\opistate}[1]{\(\boldsymbol{\theta}(#1)\)}

\newcommand{\N}{\mathbb{N}}

\newcommand{\Q}{\mathbb{Q}}
\newcommand{\R}{\mathbb{R}}